\documentclass[reqno,12pt,letterpaper]{amsart}
\usepackage{amsmath,amssymb,amsthm,graphicx,mathrsfs,url}
\usepackage[usenames,dvipsnames]{color}
\usepackage[colorlinks=true,linkcolor=Red,citecolor=Green]{hyperref}
\usepackage{amsxtra}

\newcommand{\RR}{{\mathbb R}}
\newcommand{\SP}{{\mathbb S}}

\newcommand{\CC}{{\mathbb C}}

\newcommand{\CI}{C^\infty}
\newcommand{\CIc}{C^\infty_{\comp}}

\newtheorem{theo}{Theorem}
\newtheorem{prop}{Proposition}[section]

\newtheorem{lemm}[prop]{Lemma}

\numberwithin{equation}{section}

\DeclareMathOperator{\comp}{\mathit{c}}
\DeclareMathOperator{\loc}{\mathrm{loc}}

\let\Im=\Imag

\let\epsilon=\varepsilon

\DeclareMathOperator{\supp}{supp}
\DeclareMathOperator{\tr}{tr}

\title[Heat traces and existence of resonances]{Heat traces and existence of scattering resonances for bounded potentials}
\author{Hart Smith}
\email{hart@math.washington.edu}
\address{Department of Mathematics, University of Washington, 
Seattle, WA 98195, USA}
\author{Maciej Zworski}
\email{zworski@math.berkeley.edu}
\address{Department of Mathematics, University of California,
Berkeley, CA 94720, USA}

\begin{document}

\maketitle

\begin{abstract} We show that any real valued
bounded potential with compact
support, $ V \in L^\infty_{\comp} ( \RR^n ; \RR) $, $ n $ odd, 
has at least one scattering resonance. For $ n \geq 3 $ 
this was previously known
only for sufficiently smooth potentials. The proof is based on the 
following inverse result: 
\[   V \in L^\infty_{\comp}
( \RR^n ; \RR ) , \ \  t^{\frac n 2 } \tr ( e^{  t ( \Delta - V ) } - e^{ t \Delta })
\in \CI ( [ 0 , \infty ) )  \ \Longleftrightarrow \  V \in 
\CIc ( \RR^n ; \RR) . \]
\end{abstract}

\section{Introduction and statements of results}

Let $ V \in L^\infty_{\comp} ( \RR^n ; \RR) $ be a bounded, compactly 
supported, real valued potential and let $ n \geq 3 $ be odd. 
We consider the Schr\"odinger operator, 
\begin{equation}
\label{eq:PV}  P_V=-\Delta+V(x) , 
\end{equation}
and ask the question whether $P_V$ always (for $ V \neq 0$) has infinitely 
many scattering resonances. Scattering resonances are defined
defined
as poles of the meromorphic continuation of the resolvent
\begin{equation}
\label{eq:RV}   R_V ( \lambda) :=  ( - \Delta + V - \lambda^2 )^{-1} ,  \ \ \text{ $ n$ odd, } 
\end{equation}
from $ \Im \lambda > 0 $ 
to $ \lambda \in \CC $.
These poles have many interesting interpretations and in particular 
appear in expansions of solutions to the wave equation -- see \S \ref{rst}
and references given there. For $ n $ even the situation is more
complicated as the meromorphic continuation has a logarithmic branch singularity at
$ \lambda = 0 $ -- see \cite{CH1} and references given there. Here we prove that
\begin{theo}
\label{th:resb}
Suppose that $ V \in L^\infty_{\comp } ( \RR^n ; \RR ) $ and that $ n $ is 
odd. Then the meromorphic continuation of the resolvent \eqref{eq:RV}, 
\[  R_V ( \lambda )  : L^2_{\comp}  ( \RR^n ) \to L^2_{\loc} ( \RR^n ) , \ \ 
\lambda \in \CC , \]
has at least one pole. If $\, V \in L^\infty_{\comp} ( \RR^n ; \RR ) 
\cap H^{\frac{n-3}2} ( \RR^n ) $ then $ R_V $ has infinitely many poles.
\end{theo}

For $ V \in \CIc ( \RR^n ; \RR ) $ existence of infinitely many 
resonances was proved by 
Melrose \cite{Mel} for $ n = 3 $ 
and by 
S\'a Barreto--Zworski \cite{SZ} for all odd $ n $. 
Soon afterwards quantitative statements about the
counting function, $ N ( r ) $, of resonances in $ \{ |\lambda | \leq r \}$ were
obtained by Christiansen \cite{Ch} and S\'a Barreto \cite{Sa}:
\[ \lim \sup_{ r \to \infty } \frac{ N ( r ) } { r } > 0 . \]
For potentials generic in $ \CIc ( \RR^n;  {\mathbb F} ) $ or $ L^\infty_{\comp} ( 
\RR^n ; {\mathbb F} ) $, $ {\mathbb F} = \RR $ or $ \CC $, Christiansen and Hislop
\cite{Ch0},\cite{CH} proved a stronger statement 
\begin{equation}
\label{eq:CH}  \lim \sup_{ r \to \infty } \frac{ \log N ( r ) } {\log r } = n . 
\end{equation}
This means that the upper bound $ N ( r ) \leq C r^n $ from \cite{z} is 
optimal for generic complex or real valued potentials. The only case 
of asymptotics $ \sim r^n $ for {\em non-radial} potentials was 
provided by Dinh and Vu \cite{dv13} who proved that a large class of $ L^\infty $
potentials supported in $ B ( 0 , 1 )$ has resonances satisfying a Weyl law.

To prove Theorem \ref{th:resb} we proceed by contradiction, as in \cite{bs}, \cite{Mel} and \cite{SZ}, and assume that there are no resonances.
By a direct argument (Proposition \ref{p:scatdet}) this implies
that the scattering phase is a polynomial.
This in turn implies (Proposition \ref{p:bk}) that the heat trace has an asymptotic expansion. The main result of this note, Theorem \ref{th:heat} below, 
shows that this implies that $ V \in \CIc $, and since it is real valued we obtain a contradiction by \cite{Mel} and \cite{SZ}. (We provide a direct argument of the contradiction in \S \ref{rst}.)  See \S \ref{whynot} for why our arguments do not yield a contradiction for a finite number of resonances if $n\ge 5$ and $V\in L^\infty_{\comp}(\RR^n,\RR)$.

Although
we expect \eqref{eq:CH}, or possible even $ N ( r ) > r^n / C $ when $r\gg 1$,
to be true for all non-zero real valued potentials, Christiansen 
gave classes of examples of non-zero $ V \in \CIc ( \RR^n ; \CC ) $ 
which have no resonances. Potentials $ V \in \CIc ( \RR^n ; \CC ) $ are, however, known to have infinitely many resonances if 
\begin{equation}
\label{eq:V2} \int V^2 dx \neq 0 . \end{equation}
The condition \eqref{eq:V2} arises naturally from the use of heat trace coefficients in scattering asymptotics.

Our argument outlined above depends on the following, which is the principal new result of this paper.
\begin{theo}
\label{th:heat}
Suppose that $ P_V $ is given by \eqref{eq:PV}, and $ V \in L^\infty_{\comp} 
( \RR^n ; \RR ) $, where $ n \ge 1 $ may be even or odd. If 
\begin{equation}
\label{eq:ttrace}   t^{ \frac{n}{2}} \tr \left( e^{ -t P_V } - e^{ - tP_0} \right) 
\in \CI ( [ 0 , \infty ) ) \end{equation}
then $ V \in \CIc ( \RR^n; \RR ) $. 
\end{theo}

Theorem \ref{th:heat} is a direct consequence of a more precise
result presented in Theorem \ref{thm:main} in \S \ref{heat}.
The study of heat expansions has a very long tradition going
back to Kac, Berger and McKean--Singer -- see \cite{cdv}, \cite{Gil}, \cite{PoHi}
for more recent accounts and references. Theorem \ref{th:heat}, although
not surprising, seems to be new. However, closely related inverse results
are well known. They concern recovering Sobolev norms from the
Taylor expansion coefficients of \eqref{eq:ttrace} for smooth potentials,
and using the resulting a priori bounds to 
prove compactness of sets of isospectral potentials -- see
Br\"uning \cite{Br} and Donnelly \cite{Don}, and for the origins
of that approach,  McKean--van Moerbeke \cite{MM}.

The paper is organized as follows. In \S \ref{rst} we review the
scattering theory needed for the proof of Theorem \ref{th:resb}. 
For detailed proofs we refer to the original papers and to 
the online notes \cite{res}. The section on the heat trace
\S \ref{heat} is by contrast completely self-contained. Some aspects
of the approach in \S \ref{heat} appear to be new, in particular the use of 
Gagliardo--Nirenberg--Moser inequalities in a bootstrap regularity scheme. 

\def\smallsection#1{\smallskip\noindent\textbf{#1}.}

\smallsection{Acknowledgements} We would like to thank
Gunther Uhlmann for a helpful discussion, in particular for reminding us of 
the references \cite{Br} and \cite{MM}, and Tanya Christiansen for 
helpful comments on the first version of this note. 
This material is based upon work supported by 
the National Science Foundation under Grants DMS-1161283(HS) 
and DMS-1201417(MZ).

\section{Review of scattering theory}
\label{rst}

Here we recall various facts in scattering theory and 
show how Theorem \ref{th:resb} follows from Theorem \ref{th:heat}.

\subsection{The scattering matrix}

The continued resolvent, $ R_V ( \lambda ) $, given in \eqref{eq:RV}
does not have any poles on $ \RR \setminus \{ 0 \} $ -- that 
is a well known consequence of the Rellich uniqueness theorem -- 
see \cite[\S 3.6]{res}. This implies that, for $ 
\lambda \in \RR \setminus \{ 0 \} $ and $ \omega \in \SP^{n-1} $,
there exist (unique)
solutions to 
\begin{gather}
\label{eq:outgs} 
\begin{gathered}  ( P_V - \lambda^2 ) w ( x , \lambda , \omega ) = 0 , \ \ 
w ( x, \lambda, \omega )  = e^{ - i \lambda \langle x , \omega \rangle } + u ( x , \lambda, \omega ) , 
\\ u ( x , \lambda, \omega) = |x|^{-\frac{n-1}2} e^{ i \lambda | x| } \left( b ( \lambda,  x/|x| , \omega) + 
\mathcal O ( |x|^{-1} ) \right), \ \ |x| \to \infty . \end{gathered}
\end{gather}
The {\em radiation pattern} $ b ( \lambda, \theta , \omega) $,  is the 
observed field in a scattering experiment. The {\em scattering matrix},
$ S_V ( \lambda ) $, 
can be defined using $ b ( \lambda , \theta , \omega) $. This definition 
is not the most intuitive, and we refer to \cite[\S 3.7]{res} for 
motivation. Here we define $ S_V  ( \lambda ) : L^2 ( \SP^{n-1} ) \to 
L^2 ( \SP^{n-1} ) $ as 
\begin{gather}
\label{eq:scat}
\begin{gathered}
S_V  ( \lambda ) f ( \theta ) = f ( \theta ) + \int_{\SP^{n-1} } 
a ( \lambda, \theta, \omega ) f ( \omega ) d \omega , \\
a ( \lambda, \theta , \omega ) := ( 2 \pi)^{-\frac{n-1}2} e^{  \frac\pi 4
( n - 1) i } \lambda^{\frac{n-1}2} b ( \lambda , \theta, - \omega ) . 
\end{gathered}
\end{gather}
We also have the following useful representations of $ a ( \lambda , \theta, \omega) $:
\begin{equation}
\label{eq:repa}
\begin{split}
a ( \lambda, \theta, \omega ) &= a_n \lambda^{n-2} 
\int_{\RR^n} e^{-  i \lambda \langle x , \theta \rangle } V ( x ) w ( x, \lambda, - \omega ) dx \\
 &= a_n \lambda^{n-2} 
\int_{\RR^n} e^{-  i \lambda \langle x ,  \omega - \theta\rangle }
( 1 - e^{ - i \lambda \langle x , \omega \rangle } ) R_V ( \lambda ) 
( e^{ \lambda \langle \bullet, \omega \rangle} V) ( x )  dx , 
\end{split} 
\end{equation}
where $ a_n = ( 2 \pi)^{-n+1}/2i $. 

The scattering matrix is unitary for $ \lambda $ real, and from 
\eqref{eq:repa} we see that it continues meromorphically to all of $ \CC$.
Hence we have
\begin{equation}
\label{eq:scatu}
S_V  ( \lambda)^{-1}  = S_V ( \bar \lambda )^* , \ \ \lambda \in \CC . 
\end{equation}
Another symmetry comes from changing $ \lambda $ to $ - \lambda $:
\begin{equation}
\label{eq:scats}
S_V ( \lambda )^{-1}  = J S_V  ( - \lambda ) J , \ \ J f ( \theta ) := f ( - \theta ) .
\end{equation}
The operator $ S_V  ( \lambda ) - I $ is of trace class, and hence
$ \det  S_V ( \lambda ) $ is well defined. The following result, see \cite[Theorem 3.4]{res} or \cite{Zw}, is important
for the investigation of scattering resonances:

\begin{prop}
\label{p:scatdet}
Suppose that $ V \in L^\infty_{\comp } ( \RR^n ; \RR ) $, where $ 
n $ is odd. Then 
$ \det S_V ( \lambda )  $ is a meromorphic function of order $ n $.
More precisely, 
\begin{equation}
\label{eq:scatdet}
\det S_V ( \lambda ) = ( - 1 )^{m} 
\left( \prod_{ k = 1}^K \frac{  i \mu_k + \lambda } { i \mu_k - \lambda } \right)
\frac{ P ( - \lambda ) }{ P ( \lambda ) }\,, 
\end{equation}
where 
 $ \mu_k \geq 0 $, 
 $ - \mu_1^2 < -\mu_{2}^2 \leq \cdots \leq -\mu_K^2 \leq 0 $ are
the eigenvalues of $ P_V $, included according to multiplicity, 
$ P ( \lambda ) $ is entire and non-zero for $ \Im \lambda \geq 0 $, 
and
\begin{equation} 
\label{eq:Pla} 
| P ( \lambda ) | \leq C_\epsilon e^{ C_\epsilon r^{n + \epsilon } } , 
\ \ \ \ \text{for any $ \epsilon > 0 $.} 
\end{equation}
\end{prop}

The power $ m $ in \eqref{eq:scatdet} 
is the multiplicity of the zero resonance, $ m = 0 $ or $ 1 $ 
for $ n =1,3$ and $ m = 0 $ for for $ n \geq 5$; see \cite[\S 3.3]{res} and \cite{JK}. 
 
We make the following observation based on the second representation 
in \eqref{eq:repa}:
\begin{equation}
\label{eq:longright}
 \text{ $ \lambda $ is a pole of $ \det S_V $ } \Longrightarrow 
\text{ $ \lambda $ is a pole of $ S_V $ } \Longrightarrow \text{ $ \lambda$
is a pole of $ R_V $.} \end{equation}
A more precise statement is possible (see \cite[Theorem 3.42]{res}) but 
we do not need it here. To show existence of poles of $ R_V $ we only 
need to show existence of poles of $ \det S_V $.

\subsection{A trace formula} 
The tool connecting the scattering matrix to the heat trace is the
Birman--Krein trace formula. In \S \ref{heat} we will recall the argument
showing that 
$ e^{-t P_V } - e^{-tP_0 } $ is of trace class. 
\begin{prop}
\label{p:bk}
Suppose that $ V \in L^\infty_{\comp} ( \RR^n ; \RR ) $. 
Then, in the notation of Proposition \ref{p:scatdet}, 
\begin{equation}
\label{eq:bk} 
\tr ( e^{- t P_V } - e^{- t P_0 } ) = 
\frac{1}{ 2\pi i } \int_0^\infty \tr \left( S_V (  \lambda )^{-1} \partial_\lambda 
S_V ( \lambda ) \right) e^{- t \lambda^2 } d \lambda  + \sum_{k=1}^K e^{ t \mu_k^2} + 
{\textstyle{\frac12} } m .
\end{equation}
\end{prop}
If $V\in C_{\comp}^\infty$, this is proved for $ n = 3 $ in \cite{cdv1}, and for $n\ge 5$ in
\cite{Gui} and references given there. The proofs for $ V \in L^\infty_{\comp} $
can be found in \cite[\S 3.8, \S 4.6]{res}.

Since $ |\det S_V ( \lambda ) | = 1 $ for $ \lambda \in \RR $  (which follows 
from \eqref{eq:scatu}, the unitarity of 
the scattering matrix) we can define
the winding number of the scattering phase: 
\[  \sigma ( \lambda ) := \frac{1}{2 \pi i } \log \det S_V ( \lambda ) , \ \  \ \  
\sigma' ( \lambda ) = \frac{1 }{ 2 \pi i }  \tr \left( S_V (  \lambda )^{-1} \partial_\lambda 
S_V ( \lambda ) \right) , \ \ \lambda \in \RR.  \]
In the case of $ V \in \CIc ( \RR^n , \RR ) $, $ n $ odd, $\sigma(\lambda)$
admits a full asymptotic expansion for $\lambda\rightarrow\infty$, with only odd powers of $ \lambda $ except
for the constant term. When $ n = 3 $,
\[  \theta ( + \infty ) - \theta ( 0 ) = K + {\textstyle{\frac12}} m , \quad
\theta ( \lambda ) := \sigma ( \lambda ) + \lambda \left(
{\textstyle{\frac12}}  \int_\RR V ( x ) dx \right) , \]
and for $ n \geq 5$, 
\[ \theta ( + \infty ) - \theta ( 0 ) = K , \quad
 \theta ( \lambda ) := \sigma( \lambda ) - \sum_{k=1}^{\frac{n-1}2}  c_k ( V ) \lambda^{n-2k}\,, \]
where $ c_k ( V ) $ are the coefficients in the expansion of $ \sigma ( \lambda ) $. For proofs see \cite{cdv1}, \cite{Gui}, \cite[\S 3.7]{res}, and for less regular potentials
but fewer expansion terms \cite{J}.

\subsection{Proof of Theorem \ref{th:resb}}
\label{pfth}
If $ V $ has no resonances
then Proposition \ref{p:scatdet}
shows that
\[  \det S_V ( \lambda ) = 
\frac{ P ( - \lambda ) }{ P ( \lambda ) }, 
\]
where  $ P ( \lambda ) $ is an entire function with no zeros and 
of order $ n $. This implies that $ P ( \lambda ) = e^{G ( \lambda ) }$
where $  G( \lambda ) $ is a polynomial of degree at most $ n $; see for instance \cite[8.24]{Tit}.
Defining the odd polynomial 
 $ g ( \lambda ) = ( G ( - \lambda ) - G ( \lambda ))/(2\pi i ) $, we obtain
\[ \det S_V ( \lambda ) =  e^{ 2 \pi i  g ( \lambda ) } ,
 \qquad
 \sigma' ( \lambda ) = g' ( \lambda )  . \]
 The unitarity of $ S_V ( \lambda ) $ for $ \lambda $ real 
 shows that $ g ( \lambda ) $ has real coefficients, 
 $ g ( \lambda ) = a_0 \lambda^n + a_{1} \lambda^{n-2} + \cdots + a_{\frac{n-1}2} \lambda $.
 Hence,
 \begin{equation}  
 \label{eq:bj0} \int_0^\infty \sigma' ( \lambda ) \, e^{ - t \lambda^2 } d \lambda 
\, = \,  t^{-\frac n 2} \sum_{j=0}^{\frac{n-1}2} a_j' t^j 
 ,\end{equation}
where $ a_n' := a_n \Gamma ( n/2 - j + 1 )$.

We now insert \eqref{eq:bj0} into the trace formula \eqref{eq:bk} to 
see that $ t^{n/2} \tr ( e^{ - t P_V } - e^{-tP_0 } ) $ has a full
asymptotic expansion $ \sum_{j=0}^\infty c_j t^j $ as $ t \to 0 + $.
That means that the assumption of Theorem \ref{th:heat} is satisfied,
and hence $ V \in \CIc ( \RR^n ; \RR ) $. But the result of \cite{SZ}
(see also \cite[\S 3.7]{res})
then contradicts 
our assumption that $ V$ has no resonances: every nonzero potential in $ \CIc(\RR^n,\RR) $
has to have infinitely many resonances. 

Christiansen's
argument \cite{Ch} that there must be at least one resonance for nonzero $ V \in \CIc ( \RR^n; \RR ) $ is simple and elegant, and 
we reproduce it here. As above, absence of resonances would imply 
that $ \sigma' ( \lambda ) = a_0' \lambda^{n-1} + a_1' \lambda^{n-3} + \cdots
a_n' $. Comparison with the heat expansion shows that
$ a_2' = c_n \int V^2 \neq 0 $. That immediately provides a contradiction 
in the case of $ n = 3$. When $ n \geq 5 $ we use the representation 
\eqref{eq:repa}:
\[ \begin{split} \sigma' ( \lambda ) & = 
\tr S_V ( \lambda)^* \partial_\lambda S_V ( \lambda ) \\
& = \int_{ \SP^{n-1} } \partial_\lambda a ( \lambda, 
\theta, \theta ) d \theta + \int_{\SP^{n-1}} \int_{\SP^{n-1}}
\overline { a ( \lambda, \omega, \theta )  } \partial_\lambda a 
( \lambda , \omega, \theta ) d \omega d \theta . \end{split} \]
Under the assumption that $ R_V $ is holomorphic, that is no poles, \eqref{eq:repa}
then shows that $ \sigma' ( \lambda ) = \mathcal O ( \lambda^{n-3} ) $ as 
$ \lambda \to 0 $. But this contradicts $ a_2' \neq 0 $, since 
that would imply a lower order of vanishing at $ \lambda = 0 $.

We now use Theorem \ref{thm:main} to show that if  
$ V \in L^\infty_{\comp}(\RR^n,\RR)\cap H^{\frac{n-3}2}(\RR^n) $, then $ R_V $ has infinitely many poles.
This is again seen by contradiction, by assuming that $ \det S_V ( \lambda ) $
has only finitely many resonances. In that case, let
$ - \mu_1^2 <  -\mu_2^2 \leq \cdots \leq -\mu_{K'}^2 < 0 , $ 
$ \mu_k > 0 $, denote the negative eigenvalues of $ P_V $, and  let $
i \rho_j $, $ \rho_j < 0 $, $ j =1 , \ldots , J_1$, 
$ \lambda_j  \neq - \bar \lambda_j$, $ j = 1, \ldots, J_2 $ 
the remaining finite set of resonances. Proposition \ref{p:scatdet} gives
\begin{equation*}
\det S_V ( \lambda ) = ( -1 )^m e^{ g ( \lambda ) }
\prod_{k=1}^{K'} \frac{ \lambda + i\mu_k }{ \lambda - i \mu_k } 
\prod_{j=1}^{J_1} \frac{ \lambda + i \rho_j } { \lambda - i \rho_j } 
\prod_{j=1}^{J_2} \frac{ \lambda - \bar \lambda_j}{\lambda - \lambda_j } 
\frac{ \lambda -  \lambda_j}{\lambda + \bar \lambda_j } .
\end{equation*}
Hence for $ \lambda \in \RR $, 
\begin{equation}
\label{eq:expsig}
\begin{split}
 \sigma' ( \lambda )  - g'  ( \lambda ) & =  - \frac{1}{\pi} \sum_{k=1}^{K'} \frac{
 \mu_j} { \lambda^2 + \mu_j^2 } 
-  \frac{1}{ \pi} \sum_{j=1}^{J_1 } \frac{ \rho_j } { \lambda^2 + \rho_j^2 } 
\\
& \ \ \ \ \  -  \frac{1}\pi \sum_{j=1}^{J_2} \left( \frac{  \Im \lambda_j} { | \lambda - \lambda_j |^2} + \frac{ \Im \lambda_j } { | \lambda +  \lambda_j|^2 } \right),
\end{split}
\end{equation}
That implies that 
\begin{equation}
\label{eq:int0}  \int_0^\infty \left(  \sigma' ( \lambda )  - g'  ( \lambda ) \right)
d\lambda = 
 \textstyle{ - \frac12 K' + \frac12 J_1 + J_2 }. \end{equation}
where $K'\le K$ is the number of negative eigenvalues.
We compare this with Proposition \ref{p:bk} and
the expansion in Theorem \ref{thm:main}: if $ V \in L^\infty_{\comp}(\RR^n,\RR) \cap 
H^{\frac{n-3}2}(\RR^n) $, then \eqref{eqn:asymptotics} shows that
\[\tr( e^{ - t P_V } - e^{-tP_0 } ) \, = \, \sum_{k=1}^{\frac{n-1}2} c_k' t^{-\frac n 2 + k} + \mathcal O ( t^{\frac12} ) . \]
In particular, 
\begin{equation}
\label{eq:int1}  
 \tr( e^{ - t P_V } - e^{-tP_0 } ) \, - \, \sum_{k=1}^{\frac{n-1}2} c_k'\, t^{-\frac n 2 + k} \rightarrow 0 , \ \  \ \ t \to 0 + \,. \end{equation}
Since the terms on the right hand side of \eqref{eq:expsig} 
make bounded contributions, comparison with \eqref{eq:bk} shows that
\[ \sum_{k=1}^{\frac{n-1}2} c_k'\, t^{ -\frac n 2 + k} = \frac{1}{ 2 \pi i } 
\int_0^\infty g' ( \lambda ) e^{ - t \lambda^2 } dt . \]
Using \eqref{eq:bk} and \eqref{eq:int0} we obtain
\[ \begin{split}  
 \tr( e^{ - t P_V } - e^{-tP_0 } ) - \sum_{k=1}^{\frac{n-1}2} c_k'\, t^{-\frac n 2 + k}  & = 
 \tr( e^{ - t P_V } - e^{-tP_0 } ) - \frac{1}{ 2 \pi i } 
\int_0^\infty g' ( \lambda )e^{-t \lambda^2} d \lambda 
\\
& = \frac 1 { 2 \pi i } 
\int_0^\infty ( \sigma' ( \lambda ) - g'( \lambda) )e^{ - \lambda^2 t } \,
d \lambda + \sum_{ k=1}^K e^{ \mu_k ^2 t } + \textstyle{\frac12}  m \,.
\end{split}
 \]
Taking the limit as $t\rightarrow 0+$ we obtain
\[
\int_0^\infty ( \sigma' ( \lambda ) - g'( \lambda) ) \,
d \lambda + K + {\textstyle{\frac12}} m 
\, = \, K - {\textstyle{\frac12}} K' + {\textstyle{\frac12}} m + 
{\textstyle{\frac12}}  J_1 + J_2 > 0 \,.
\]
But this contradicts \eqref{eq:int1}.

\subsection{Why not infinitely many?}
\label{whynot}
A frustrating aspect of the argument in \S \ref{pfth} is that for 
$ V \in L^\infty_{\comp}(\RR^n,\RR) $, $ n \ge 5 $,
it only shows existence of {\em one} resonance. The reason for that is
the strong assumption in Theorem \ref{th:heat}. If we allowed, for example, a unique (non-zero) resonance $ \lambda_0 = i \rho_0 $ 
(it has to be purely imaginary, as the symmetry $ \lambda 
\mapsto - \bar \lambda $ would otherwise imply that there are two)
then the factorization argument above would imply
\[  \det S_V ( \lambda ) = e^{ 2 \pi i g (\lambda ) } 
\frac{ i \rho_0 + \lambda }{ i \rho_0 - \lambda } , 
\qquad \sigma' ( \lambda ) = g ( \lambda ) - \frac 1 \pi \frac{ \rho_0} { \lambda^2 + \rho_0^2} .
\]
We now note that 
\begin{equation}
\label{eq:bj}
\frac{1}{\pi} \int_0^\infty \frac{ e^{-s r^2} }{ 1 + r^2 } dr 
\sim {\textstyle{\frac12}} e^{ s } + s^{\frac 12} \sum_{j=0}^\infty b_j s^j , \ \ s \to 0 +. 
\end{equation}
To see \eqref{eq:bj}, let $ I( s) := (1/\pi) \int_0^\infty e^{ -s ( 1 + r^2) }/
( 1 + r^2 ) dr $. Then the right hand side of \eqref{eq:bj} is $
e^{s} I ( s ) $, while $ \partial_s I ( s ) = - (1/\pi) \int_0^\infty 
e^{ - s ( 1 + r^2 )} d r = \alpha e^{-s} / s^{\frac12} $, $ \alpha = 1/2\sqrt \pi $.
Hence $ I ( s ) = I ( 0 ) + \alpha \int_0^s e^{-s_1} s_1^{-\frac12} \sim 
\frac12 +  s^{\frac12} \sum_{ j= 0}^\infty b_j' s^j $. Multiplying by $ e^s $ gives \eqref{eq:bj}.

Inserting \eqref{eq:bj} into the trace formula \eqref{eq:bk}, and 
noting that if $ \rho_0 > 0 $ we have an eigenvalue, gives
\[  \tr ( e^{- t P_V } - e^{ - t P_0 } ) = t^{-n/2} \, \sum_{ j=1}^\infty 
a_j t + \textstyle{\frac{1}2} e^{ \rho_0^2 t } , \]
and we {\em cannot} use Theorem \ref{th:heat} to conclude that $ V $ is
smooth. The same problem arises if we assume that we have two (or more)
resonances, $\lambda_0 $, $  - \bar \lambda_0 $. 

The following simple example does not fit into our hypotheses, but it suggests
a possible complication. Consider $ n =1 $ and $ V = \delta_0 $. Then 
there is only one resonance, at $ \lambda = -2 i $, and the heat 
trace has an expansion with both integers and half-integers.

\section{Heat trace expansions}
\label{heat}

For $ P_V $ given by 
\eqref{eq:PV} with $ V \in L^\infty_{\comp } ( \RR^n ; \CC ) $, it is well known that $e^{-tP_V}-e^{-tP_0}$ is trace class for $t>0$, and if $V\in C_{\comp}^\infty$ it is known that
$\tr\bigl(e^{-tP_V}-e^{-tP_0}\bigr)$ admits a full asymptotic expansion 
-- see for instance \cite{vdb} and references given there.

Theorem \ref{th:heat} is a consequence of a converse result 
that gives a sharp relation between existence of 
a finite expansion for the trace, and a given finite order of Sobolev regularity for $V$, assuming that $V$ is real-valued.

\begin{theo}\label{thm:main}
Suppose that $V\in L^\infty_{\comp}(\RR^n,\RR)$, and that for some
 $m\in\mathbb N$ one can write
\begin{equation}\label{eqn:asymptotics}
\tr\bigl(e^{-tP_V}-e^{-tP_0}\bigr)=(4\pi t)^{-\frac n2}
\Bigl(c_1 t+c_2 t^2+\cdots+c_{m+1}t^{m+1}+r_{m+2}(t)t^{m+2}\Bigr)\,,
\end{equation}
where $|r_{m+2}(t)|\le C$ for $0<t\le 1$. Then $V\in H^m(\RR^n)$. Conversely, if $V\in H^m(\RR^n)$ then \eqref{eqn:asymptotics} holds with such an 
$r_{m+2}(t)$, and $\lim_{t\rightarrow 0^+}r_{m+2}(t)=c_{m+2}$ exists.
\end{theo}

The proof of Therem \ref{thm:main} begins by using iteration to expand the heat kernel for $P_V=-\Delta+V$. The formula is
$$
e^{-tP_V}-e^{-tP_0}=\sum_{k=1}^\infty W_k(t)\,,
$$
where
\[
W_k(t)=\int_{0<s_1<\cdots\,<s_k<t}e^{-(t-s_k)P_0}\,V\,e^{-(s_k-s_{k-1})P_0}\,V\,\cdots\,V\,e^{-(s_2-s_1)P_0}\,V\,e^{-s_1P_0}ds_1\cdots ds_k\,.
\]
Convergence of the expansion in the $L^2$ operator norm  follows from $\|W_k(t)\|_{L^2\rightarrow L^2}\le \|V\|_{L^\infty}^kt^k/k!$, which holds since for all $s_j$ and $t$ the integrand is $L^2$ bounded by $\|V\|_{L^\infty}^k$, and the volume of integration is $t^k/k!$. 

We also have a bound on the trace class norm:
\begin{equation}
\label{eq:Wk-trace} 
 \|W_k(t)\|_{\mathcal L^1}\le C^k\,k^{\frac n2}\,t^{k-\frac n2}/k!, 
 \end{equation}
where $n$ is the dimension. For this we use that the trace class is an ideal, so it suffices to show that one pair of successive terms in the product has $\mathcal L^1$ bound
less than $C\,k^{\frac n2}\,t^{-\frac n2}$. We then observe that at least one of $t-s_k$, $s_{j+1}-s_j$ or $ s_1 $ is greater than $t/k$, and for that term we use the trace bound 
\begin{equation}
\label{eq:tracetrivial} 
 \|e^{-sP_0}\chi\|_{\mathcal L^1}\le C \,s^{-n/2} , 
 \end{equation}
 where $\chi\in \CIc$ is chosen to be $1$ on the support of $V$.
 To prove \eqref{eq:tracetrivial} we choose $ \chi_1 \in \CIc $
 equal to $ 1 $ on $ \supp \chi $. Then the explicit Schwartz kernel, $ K_1 ( x , y ) $ of 
 $ ( 1 - \chi_1 ) e^{-sP_0} \chi $ satisfies $ | \partial^{ \alpha} K_1 | 
 \leq C_{\alpha, N} s^N ( 1 + |x| + |y|)^{-N} $, for any $ \alpha $ and $ N$.
 Hence $\| ( 1 - \chi_1 ) e^{sP_0} \chi \|_{\mathcal L^1} = \mathcal O ( s^\infty )$.
 On the other hand, if $ K_2 ( x,y ) $ is the Schwartz kernel of $ e^{-sP_0/2} \chi_1 $
 then $ \int\!\!\int \!|K_2 ( x, y )|^2 dx dy \leq C s^{-n/2} $ which provides
an estimate $ \mathcal O ( s^{-n/4} ) $ on the Hilbert--Schmidt norm. These two 
bounds give 
 \eqref{eq:tracetrivial}:
 \[ \begin{split}  \|e^{-sP_0}\chi\|_{\mathcal L^1} & \le   C \| \chi_1 e^{ - s P_0} \chi_1 \|_{\mathcal L^1} + \| ( 1 - \chi_1 ) e^{ -sP_0} \chi \|_{\mathcal L ^1 } \\
 & \leq 
C\| \chi_1 e^{ -sP_0/2} \|_{ \mathcal L^2 } ^2  +  C_N s^N \leq 
  C \,s^{-n/2}. \end{split} \]

Using \eqref{eq:Wk-trace}, we see that 
$ e^{ - t P_V } - e^{ -t P_0 } $ is of trace class for $ t> 0 $.
The trace can be brought into the sum, and we write
$$
\tr\Bigl(e^{-tP_V}-e^{-tP_0}\Bigr)=\sum_{k=1}^\infty\tr\bigl(W_k(t)\bigr)\,.
$$
It is well known, and we include the proof, that
$$
\tr\bigl(W_1(t)\bigr)=(4\pi t)^{-\frac n2}\,t\,\int V(y)\,dy\,,
$$
which shows that $c_1=\int V$, and the expansion \eqref{eqn:asymptotics} is equivalent to
$$
\sum_{k=2}^\infty\tr\bigl(W_k(t)\bigr)=(4\pi t)^{-\frac n2}
\Bigl(c_2 t^2+\cdots+c_{m+1}t^{m+1}+r_{m+2}(t)\Bigr)\,.
$$
Theorem \ref{thm:main} will then follow as a result of the following two propositions that concern the asymptotics of the individual terms $\tr\bigl(W_k(t)\bigr)$.

\begin{prop}\label{thm:traceW2}
If $V\in L^\infty_{\comp}(\RR^n,\RR)\cap H^m(\RR^n)$, then one can write
\begin{equation}\label{eqn:traceW2}
\tr\bigl(W_2(t)\bigr)=(4\pi t)^{-\frac n2}
\Bigl(c_{2,2} t^2+\cdots+c_{2,2+m}t^{2+m}+\varepsilon(t)t^{2+m}\Bigr)\,,
\end{equation}
with $\lim_{t\rightarrow 0^+}\epsilon(t)=0$ and 
$c_{2,2+j}=a_j \||D|^jV\|_{L^2}$ for $0\le j\le m$,
for constants $ a_j \ne 0 $.

Conversely, assuming $V\in L^\infty_{\comp}(\RR^n,\RR)\cap H^{m-1}(\RR^n)$, if one can write
\begin{equation}\label{eqn:traceW2'}
\tr\bigl(W_2(t)\bigr)=(4\pi t)^{-\frac n2}
\Bigl(c_{2,2} t^2+\cdots+c_{2,1+m}t^{1+m}+r_{2,2+m}(t)t^{2+m}\Bigr)\,,
\end{equation}
where $|r_{2,2+m}(t)|\le C$ for $0<t\le 1$, then $V\in H^{m}(\RR^n)$, and hence \eqref{eqn:traceW2} holds.
\end{prop}

\begin{prop}\label{thm:traceWk}
If $V\in L^\infty_{\comp}(\RR^n,\RR)\cap H^m(\RR^n)$, then for $k\ge 3$ one can write
\begin{equation}\label{eqn:traceWk}
\tr\bigl(W_k(t)\bigr)=(4\pi t)^{-\frac n2}
\Bigl(c_{k,k} t^k+\cdots+c_{k,k+m-1}t^{k+m-1}+r_{k,k+m}(t)t^{k+m}\Bigr)\,,
\end{equation}
where, for a constant $C$ depending on $k$ and $m$, for $0\le j\le m$,
$$
|c_{k,k+j}|\le C\,\|V\|_{L^\infty}^{k-2}\|V\|_{H^j}^2\,,\qquad
\sup_{0<t<1}|r_{k,k+m}(t)|\le C\,\|V\|_{L^\infty}^{k-2}\|V\|_{H^m}^2\,.
$$
\end{prop}

The fact that $V\in L^\infty_{\comp}(\RR^n,\RR)\cap H^m(\RR^n)$ implies existence of the asymptotic expansion \eqref{eqn:asymptotics} of order $m+2$ is an easy consequence of the above propositions. By the bound $\|W_k(t)\|_{\mathcal L^1}\le C^k\,k^{\frac n2}\,t^{k-\frac n2}/k!$ we have that
\begin{equation}\label{eqn:traceerror}
\tr\sum_{k=m+3}^\infty W_k(t)\le C\,t^{m+3-\frac n2}\,,\qquad 0<t\le 1\,.
\end{equation}
On the other hand, Propositions \ref{thm:traceW2} and \ref{thm:traceWk} show that 
$$
\tr\sum_{k=1}^{m+2} W_k(t)=(4\pi t)^{-\frac n2}
\Bigl(c_1 t+c_2 t^2+\cdots+c_{m+1}t^{m+1}+c_{m+2}t^{m+2}+\varepsilon(t)t^{m+2}\Bigr)\,,
$$
where for $j\ge 2$ we have 
$
c_j=\sum_{k=2}^{j}c_{k,j}
$.
 
The other direction of Theorem \ref{thm:main}, that existence of an asymptotic expansion implies regularity, is carried out by induction. 
Assume $m\ge 1$ and $V\in L^\infty_{\comp}\cap H^{m-1}(\RR^n)$, which trivially holds if $m=1$ since $L^\infty_{\comp}\subset L^2(\RR^n)$. Assume
\eqref{eqn:asymptotics} holds. By \eqref{eqn:traceerror} this implies
$$
\tr\sum_{k=2}^{m+2} W_k(t)=(4\pi t)^{-\frac n2}
\Bigl(c_1 t+c_2 t^2+\cdots+c_{m+1}t^{m+1}+r_{m+2}(t)t^{m+2}\Bigr)\,,
$$
where $|r_{m+2}(t)|\le C$.

By Proposition \ref{thm:traceWk}, since $V\in L^\infty_{\comp}\cap H^{m-1}(\RR^n)$ the same relation holds, with different coefficients that can be bounded from $L^\infty$ and $H^j$ norm bounds for $V$ with $j\le m-1$, for $\tr\sum_{k=3}^{m+2}W_k(t)\,.$ Hence
the relation \eqref{eqn:traceW2'} holds, and we conclude $V\in H^m(\RR^n)$.

\subsection{Calculating $\tr \bigl(W_1(t)\bigr)$.}
We calculate the trace of $W_1(t)$ by integrating over the diagonal
$$
\tr\bigl(W_1(t)\bigr)=(4\pi)^{-n}\int_{\RR^n}\int_{\RR^n}\int_0^t
 (t-s)^{-\frac n2}s^{-\frac n2}
e^{-\frac {|x-y|^2}{4(t-s)}}\,V(y)\,e^{-\frac {|y-x|^2}{4s}}\,ds\,dx\,dy\,.
$$
The integral $dx$ is carried out
$$
\int_{\RR^n} e^{-\frac{|x-y|^2}4 \frac{t}{(t-s)s}}\,dx=(4\pi)^{\frac n2}t^{-\frac n2}(t-s)^{\frac n2}s^{\frac n2}
$$
leading to
$$
\tr\bigl(W_1(t)\bigr)=(4\pi t)^{-\frac n2}\,t\,\int V(y)\,dy\,.
$$
(From now on the integrals without integration limits will denote
integrals over $ \RR^n $.)

\subsection{Calculating $\tr \bigl(W_2(t)\bigr)$.}
Again we integrate over the diagonal to write $\tr\bigl(W_2(t)\bigr)$ as
$$
(4\pi)^{-\frac{3n}2} \int_{0<r<s<t}(t-s)^{-\frac n2}(s-r)^{-\frac n2}r^{-\frac n2}
e^{-\frac {|x-y|^2}{4(t-s)}-\frac{|y-z|^2}{4(s-r)}-\frac {|z-x|^2}{4r}}\,V(y)\,V(z)
\,dr\,ds\,dx\,dy\,dz\,.
$$
We let $u=t-s$ and $x_0=\bigl(\frac r{r+u}\bigr)\,y+\bigl(\frac u{r+u}\bigr)\,z$ and carry out the integral over $x$ by writing
\begin{equation}\label{quadident}
\frac {|x-y|^2}u+\frac {|z-x|^2}r=\frac {r+u}{ru}\,|x-x_0|^2+\frac 1{r+u}\,|y-z|^2
\end{equation}
which expresses $\tr\bigl(W_2(t)\bigr)$ as
$$
(4\pi)^{-n} \int_{\substack{r+u<t \\ 0<r,u}}
 (t-u-r)^{-\frac n2}(u+r)^{-\frac n2} e^{-{\frac{|y-z|^2}4}\bigl( \frac 1{t-u-r}+\frac 1{u+r} \bigr)}\,V(y)\,V(z)
\,dr\,du\,dy\,dz\,.
$$
Let $r=tv-u$, so $dr\,du=t\,dv\,du$, the integrand is then independent of $u$, the new limits are $0<u<tv$, $0<v<1$, and we get
$$
t^2\,(4\pi t)^{-n}\int\int\int_0^1
(1-v)^{-\frac n2}v^{-\frac n2+1}\, 
e^{-{\frac{|y-z|^2}{4t}}\frac 1{v(1-v)}}\,V(y)\,V(z)
\,dv\,dy\,dz\,.
$$
Since $V$ is real we can use the Plancherel theorem to write this as
$$
t^2\,(4\pi t)^{-\frac n2}\int_0^1 v\,\biggl((2\pi)^{-n}\int e^{-t(1-v)v|\xi|^2}\,\bigl|\widehat V(\xi)\bigr|^2\,d\xi\,\biggr)\,dv\,.
$$
By symmetry under $v\rightarrow 1-v$ we can also write this as
$$
\frac 12 \, t^2\,(4\pi t)^{-\frac n2}\int_0^1 \biggl((2\pi)^{-n}\int e^{-t(1-v)v|\xi|^2}\,\bigl|\widehat V(\xi)\bigr|^2\,d\xi\,\biggr)\,dv\,.
$$
The term in parentheses is continuous in $t$, and at $t=0$ equals $\|V\|_{L^2}^2$,
so
\begin{equation}
\label{eq:W21}
\tr\bigl(W_2(t)\bigr)=\frac 12\,t^2\,(4\pi t)^{-\frac n2}\,\Bigl(\,\|V\|_{L^2}^2+ \varepsilon(t)\Bigr)\,, \qquad \lim_{t\rightarrow 0^+} \varepsilon(t)=0\,.
\end{equation}
This settles the case $m=0$ of Theorem \ref{thm:main} which, since $L^\infty_{\comp}\subset H^0(\RR^n)=L^2(\RR^n)$, is nontrivial only for the existence of the
expansion \eqref{eqn:asymptotics} for $m=0$. It also shows that we can recover $\|V\|_{L^2}$ from $\lim_{t\rightarrow 0^+}r_2(t)$.

\noindent
{\bf Remark.}
If we were to assume $V$ is H\"older-$\alpha$, then to get van den Berg's bounds \cite{vdb} we would write 
$$
V(y)V(z)=\frac 12\Bigl(V(y)^2+V(z)^2-\bigl(V(y)-V(z)\bigr)^2\Bigr)
$$
and writing $|V(y)-V(z)|^2\le |y-z|^{2\alpha}$ would lead to a gain of $t^\alpha$ for the last term on the right; the other two terms would lead to the desired leading term, so we would get $\varepsilon(t)\lesssim t^\alpha$.

\subsection{Proof of Proposition \ref{thm:traceW2}}

First consider the case $m=1$, and suppose that we have an expansion
$$
\tr\bigl(W_2(t)\bigr)=
(4\pi t)^{-\frac n2}\bigl(c_2\,t^2+\mathcal{O}(t^3)\bigr)\,,\quad t\le 1\,.
$$
From \eqref{eq:W21} we must have $c_2=\frac 12\|V\|_{L^2}^2$.
This leads to the estimate
$$
\int_0^1\int \biggl(\frac{1-e^{-t(1-v)v|\xi|^2}}{t}\biggr)\,\bigl|\widehat V(\xi)\bigr|^2\,d\xi\,dv\le C\,, \qquad 0<t\le 1\,.
$$
The integrand is positive, so by Fatou's lemma we get
$$
\biggl(\int_0^1 (1-v)v\,dv\biggr)\int |\xi|^2\,\bigl|\widehat V(\xi)\bigr|^2\,d\xi\le C\,,
$$
implying that $V\in H^1(\RR^n)$. Conversely, if $V\in H^1(\RR^n)\cap L^\infty_{\comp}(\RR^n,\RR)$ we would get such an expansion by dominated convergence.

To consider higher values of $m$, write
\begin{equation}\label{eqn:expexp}
e^{-s}\;=\;\sum_{j=0}^{m-1}\frac{(-1)^j}{j!}\,s^j\,+\,r_m(s)\,\frac{(-1)^m}{m!}\,s^m\,,
\end{equation}
where $r_m(s)$ is a smooth function, and by the Lagrange form for the remainder,
\begin{equation}\label{errorterm}
0\le r_m(s)\le 1\;\;\;\text{if}\;\;\;s\ge 0\,,\qquad
r_m(0)=1\,,\qquad\partial_sr_m(0)=\frac{-1}{m+1}\,.
\end{equation}
Now suppose that $V\in H^m(\RR^n)$ for some $m\ge 1$. Then we can expand
\begin{multline*}
\int_0^1 \biggl(\int e^{-t(1-v)v|\xi|^2}\,\bigl|\widehat V(\xi)\bigr|^2\,d\xi\,\biggr)\,dv=
\sum_{j=0}^{m-1}\biggl(\frac 1{j!}\int_0^1 (1-v)^jv^j\,dv\biggr)
\biggl(\int |\xi|^{2j}\,\bigl|\widehat V(\xi)\bigr|^2\,d\xi\biggr)\,t^j\\
+\frac{(-1)^m}{m!}\biggl(\int_0^1\int r_m\bigl(t(1-v)v|\xi|^2\bigr)(1-v)^mv^m\,|\xi|^{2m}\bigl|\widehat V(\xi)\bigr|^2\,d\xi\,dv\biggr)\,t^m\,.
\end{multline*}
The coefficient of $t^m$ is continuous in $t$, and converges to $a_m\,\||D|^mV\|_{L^2}^2$ as $t\rightarrow 0$, where $a_m\ne 0$. Thus, if we can write
$$
\tr\bigl(W_2(t)\bigr)=(4\pi t)^{-\frac n2}\biggl(\;\sum_{j=0}^m\,c_j\,t^j+\mathcal{O}\bigl(t^{m+1}\bigr)\,\biggr)\,,\quad t\le 1\,,
$$
then $c_j=a_j\||D|^jV\|_{L^2}^2$ for $0\le j\le m$, and in addition we have
uniform bounds for $0<t\le 1$
$$
\int_0^1\int\biggl(\frac{1-r_m\bigl(t(1-v)v|\xi|^2\bigr)}{t}\biggr)\,(1-v)^mv^m|\xi|^{2m}\bigl|\widehat V(\xi)\bigr|^2\,d\xi\,dv\le C\,.
$$
Then by Fatou's lemma and \eqref{errorterm} we get
$$
\frac{1}{m+1}\biggl(\int_0^1 (1-v)^{m+1}v^{m+1}\,dv\biggr)\biggl(\int|\xi|^{2(m+1)}\bigl|\widehat V(\xi)\bigr|^2\,d\xi\,\biggr)\le C\,,
$$
so necessarily $V\in H^{m+1}(\RR^n)$, completing the proof of Proposition \ref{thm:traceW2}.

\subsection{Trace of $W_k(t)$ for $k\ge 3$}
To estimate products of derivatives, we will use the following particular case of the Gagliardo--Nirenberg--Moser inequalities.
\begin{lemm}\label{lem:GNM}
Suppose $\{\alpha_j\}_{j=1}^k$ are multi-indices, with $|\alpha_j|\le m$, and $\sum_j|\alpha_j|=2m$. If $u\in L^\infty(\RR^n)\cap H^m(\RR^n)$, then for a constant $C$ depending only on $n$ and $m$,
$$
\bigl\|\prod_{j=1}^k \bigl(\partial^{\alpha_j}u_j\bigr)\bigr\|_{L^1}\le C\,
\Bigl(\;\sum_{j=1}^k \|u_j\|_{L^\infty}\Bigr)^{k-2}
\Bigl(\;\sum_{j=1}^k \|D^mu_j\|_{L^2}\Bigr)^2\,.
$$
\end{lemm}
\begin{proof}
We use the following bound \cite[(3.17) in \S 13.3]{pde}. Assuming $u\in L^\infty\cap H^m$,
$$
\|\partial^{\alpha_j}u_j\|_{L^{\frac {2m}{|\alpha_j|}}}\le C\,
\|u_j\|_{L^\infty}^{1-\frac{|\alpha_j|}m}\|D^m u_j\|_{L^2}^{\frac{|\alpha_j|}m}\,.
$$
The result follows by H\"older's inequality after taking the product over $j$.
\end{proof}

We now write $\tr\bigl(W_k(t)\bigr)$ for $t>0$ as an integral
$$
\int_{0<s_1<\cdots\,<s_k<t}\frac{e^{-\frac{|x-y_k|^2}{4(t-s_k)}-\frac{|y_k-y_{k-1}|^2}{4(s_k-s_{k-1})}\cdots-\frac{|y_1-x|^2}{4s_1}}
\,V(y_k)\cdots V(y_1)}
{(4\pi)^{\frac n2(k+1)}(t-s_k)^{\frac n2}\cdots(s_2-s_1)^{\frac n2}(s_1)^{\frac n2}}\,dy_1\cdots dy_k\,ds_1\cdots ds_k\,dx\,.
$$
After integrating over $x$, and letting $s_j=tr_j$, then letting $\Sigma\subset\RR^k$ denote the set $\{r\in\RR^k\,:\,0<r_1<\cdots\,<r_k<1\}$, we obtain
$$
t^k\int_\Sigma\int_{(\RR^n)^k}\frac{e^{-\frac{|y_k-y_{k-1}|^2}{4t(r_k-r_{k-1})}\cdots-\frac{|y_2-y_1|^2}{4t(r_2-r_1)}-\frac{|y_1-y_k|^2}{4t(1+r_1-r_k)}}
\,V(y_k)\cdots V(y_1)}
{(4\pi t)^{\frac n2 k}(r_k-r_{k-1})^{\frac n2}\cdots(r_2-r_1)^{\frac n2}(1+r_1-r_k)^{\frac n2}}\,dy\,dr\,.
$$

To analyse this, we introduce variables $u_1=y_1$, and $u_j=y_j-y_1$ for $2\le j\le k$.
Then $du_1\wedge\cdots\wedge du_k=dy_1\wedge\cdots\wedge dy_k$, so the formula for $\tr\bigl(W_k(t)\bigr)$ becomes
\begin{equation}\label{eqn:trWk}
\frac{t^k}{(4\pi t)^{\frac n2}}\int_\Sigma\int_{(\RR^n)^k}
G_{r,t}(u')\,\,V(u_1+u_k)\cdots V(u_1+u_2)V(u_1)\,du\,dr\,,
\end{equation}
where $G_{r,t}(u')$ is the Gaussian function of 
$u'=(u_2,\ldots,u_k)\in(\RR^n)^{k-1}$
$$
G_{r,t}(u_2,\ldots,u_k)=\frac{e^{-\frac 1{4t}\bigl(\frac{|u_k|^2}{1+r_1-r_k}+\frac{|u_k-u_{k-1}|^2}{r_k-r_{k-1}}\,\cdots\,+
\frac{|u_3-u_2|^2}{r_3-r_2}+\frac{|u_2|^2}{r_2-r_1}\bigr)}}
{(4\pi t)^{\frac n2 (k-1)}(1+r_1-r_k)^{\frac n2}(r_k-r_{k-1})^{\frac n2}\cdots(r_2-r_1)^{\frac n2}}\,.
$$
Applying successively the following equality, which is a special case of \eqref{quadident},
\begin{multline*}
\frac{|u_{j+1}-u_j|^2}{r_{j+1}-r_j}\,+\,\frac{|u_j|^2}{r_j-r_1}\\
=
\frac{r_{j+1}-r_1}{(r_{j+1}-r_j)(r_j-r_1)}\,\Bigl|\,u_j-\frac{r_j-r_1}{r_{j+1}-r_1}\,u_{j+1}\,\Bigr|^2+\;\frac 1{r_{j+1}-r_1}\,|u_{j+1}|^2
\end{multline*}
we can write the quadratic term in the exponent of $G_{r,t}$ as
\begin{equation}\label{eqn:quadform}
\frac{|u_k|^2}{(1+r_1-r_k)(r_k-r_1)}\,+\sum_{j=2}^{k-1}\frac{(r_{j+1}-r_1)}{(r_{j+1}-r_j)(r_j-r_1)}\,\Bigl|\,u_j-\frac{r_j-r_1}{r_{j+1}-r_1}\,u_{j+1}\,\Bigr|^2
\end{equation}
In particular we see that, for all $t>0$ and $r\in\Sigma$,
$$
\int_{(\RR^n)^{k-1}} G_{r,t}(u')\,du'=1\,.
$$

For $t>0$ consider the $k$-linear form
$$
B_t(V_1,\ldots,V_k)=\int_\Sigma\int_{(\RR^n)^k}
G_{r,t}(u')\,\,V_k(u_1+u_k)\cdots V_2(u_1+u_2)V_1(u_1)\,du\,dr\,.
$$
By H\"older's inequality applied to the integral over $u_1$, we have 
$$
|B_t(V_1,\ldots,V_k)|\le \prod_{j=1}^k\|V_j\|_{L^k(\RR^n)}\,,
$$
and thus $B_t$ is uniformly continuous on bounded sets in $L^k(\RR^n)^k$. The quadratic form \eqref{eqn:quadform} is bounded below by $c\,|u'|^2$, for $c>0$ independent of $r\in\Sigma$. An approximation to the identity argument then shows that $B_t$ is continuous over $t\in [0,\infty)$, for fixed elements of $L^k(\RR^n)^k$, where we set
$$
B_0(V_1,\ldots,V_k)=\frac 1{k!}\int_{\RR^n} V_k(u_1)\cdots V_1(u_1)\,du_1\,.
$$
Consequently, we can write
$$
\tr\bigl(W_k(t)\bigr)=(4\pi t)^{-\frac n2}\,t^k\,B_t(V)\,,\qquad
B_t(V)\in C\bigl([0,\infty)\bigr)\,,\quad B_0(V)=\frac 1{k!}\int V(y)^k\,dy\,.
$$
Here we set $B_t(V)=B_t(V,\ldots,V)$, which, by the above, is for each $t$ a continuous function of $V\in L^k(\RR^n)$.

We start by demonstrating an $m$-th order expansion of $B_t(V)$ when $V\in C_{\comp}^\infty(\RR^n,\RR)$, after which we will show it applies to $V\in L^\infty_{\comp}(\RR^n,\RR)\cap H^m(\RR^n)$ by taking limits.

For $2\le j\le k$ we write
$$
V(u_j+u_1)=(2\pi)^{-n}\int e^{i\eta_j\cdot(u_1+u_j)}\widehat{V}(\eta_j)
$$
and plug this into \eqref{eqn:trWk} to express 
$$
B_t(V)=(2\pi)^{-n(k-1)}\int_\Sigma\int_{(\RR^n)^{k-1}}
e^{-tQ_r(\eta')}
\widehat V(\eta_k)\cdots \widehat V(\eta_2)\,\overline{\widehat V(\eta_2+\cdots+\eta_k)}
\,d\eta_2\cdots d\eta_k\,dr\,.
$$
where $Q_r(\eta')$ is the quadratic form inverse to \eqref{eqn:quadform}, and where $\widehat V(-\zeta)=\overline{\widehat V(\zeta)}$ since $V$ is real valued.

We expand $\exp\bigl(-tQ_r(\eta')\bigr)$ as in \eqref{eqn:expexp}. The first $m-1$ terms give contributions to $B_t(V)$ of the form
$$
(2\pi)^{-n(k-1)}\sum_{j=0}^{m-1} \frac{(-1)^j}{j!}\,t^j \int\,{Q(\eta')}^j\,
\widehat V(\eta_k)\cdots \widehat V(\eta_2)\,\overline{\widehat V(\eta_2+\cdots+\eta_k)}
\,d\eta_2\cdots d\eta_k\,,
$$
where $Q(\eta')$ is the quadratic form obtained by integrating $Q_r(\eta')$ over $r$. The key observation we need is that we can write
$$
Q(\eta')^j=\sum C_{\alpha_k,\ldots,\alpha_1}\,\eta_k^{\alpha_k}\,\cdots\,\eta_2^{\alpha_2}(\eta_2+\cdots+\eta_k)^{\alpha_1}
$$
where $\sum_{i=1}^k|\alpha_i|=2j$, and $|\alpha_i|\le j$ for every $i$.

Thus, the coefficient of $t^j$ is such a linear combination of terms of the form
$$
(2\pi)^{-n(k-1)}\int \widehat {(\partial^{\alpha_k}V)}(\eta_k)\cdots \widehat{(\partial^{\alpha_2}V)}(\eta_2)\,\overline{\widehat{(\partial^{\alpha_1}V)}(\eta_2+\cdots+\eta_k)}
\,d\eta_2\cdots d\eta_k\,,
$$
This integral is equal to
$$
\int (\partial^{\alpha_k}V)(y)\cdots (\partial^{\alpha_2}V)(y)\,(\partial^{\alpha_1}V)(y)\,dy\,,
$$
which by Lemma \ref{lem:GNM} is bounded by $C\,\|V\|_{L^\infty}^{k-2}\|D^jV\|_{L^2}^2\,.$ This establishes the bounds of Proposition \ref{thm:traceWk} on the coefficients $c_{k,j+k}$, provided $V\in C_{\comp}^\infty(\RR^n)$.

The $m$-th order remainder is a constant times
$$
t^m
\int_0^1 (1-s)^{m-1}\int_\Sigma\int_{(\RR^n)^{k-1}} e^{-stQ_r(\eta')}
{Q_r(\eta')}^m\,
\widehat V(\eta_k)\cdots \widehat V(\eta_2)\,\overline{\widehat V(\eta_2+\cdots+\eta_k)}
\,d\eta'\,dr\,ds\,,
$$
which by a similar argument can be written as an integral over $r$ and $s$ of various polynomials in $r,s$ times
$$
t^m\int e^{-stQ_r(\eta')}\widehat {(\partial^{\alpha_k}V)}(\eta_k)\cdots \widehat{(\partial^{\alpha_2}V)}(\eta_2)\,\overline{\widehat{(\partial^{\alpha_1}V)}(\eta_2+\cdots+\eta_k)}
\,d\eta_2\cdots d\eta_k\,,
$$
with $|\alpha_i|\le m$, and $\sum_i|\alpha_i|=2m$.
We now show that, uniformly over $r\in\Sigma$, and $t>0$,
\begin{equation}\label{eqn:prodbound}
\frac 1{(2\pi)^{n(k-1)}}\biggl|\,\int e^{-tQ_r(\eta')}\widehat{v_k}(\eta_k)\cdots \widehat{v_2}(\eta_2)\,\overline{\widehat{v_1}(\eta_2+\cdots+\eta_k)}
\,d\eta'\;\biggr|
\le \prod_{j=1}^k\|v_j\|_{L^{p_j}}\,,
\end{equation}
whenever $2\le p_j\le \infty$ and $\sum_j p_j^{-1}=1\,.$ We note that the proof of Lemma \ref{lem:GNM} bounds the right hand side of \eqref{eqn:prodbound}, with $p_j=2m/|\alpha_j|$ and $v_j=\partial^{\alpha_j} V$, by
$\|V\|_{L^\infty}^{k-2}\|V\|^2_{H^m}$. The bounds on $r_{k,k+m}(t)$ in Proposition \ref{thm:traceWk} will then follow for $V\in C_{\comp}^\infty(\RR^n)$.

The left hand side of \eqref{eqn:prodbound} equals
$$
\biggl|\,\int G_{r,t}(y_2-x,\ldots,y_k-x)\,v_k(y_k)\cdots v_2(y_2)\,v_1(x)\,dx\,dy_2\cdots dy_k\;\biggr|\,.
$$
The kernel $G_{r,t}$ is positive and has total integral 1, so for proving the bound we may assume each $v_j$ is nonnegative. By interpolation, we may restrict to the case that two of the $p_j$'s are equal to 2, and the rest equal $\infty$. There are then two distinct cases to consider: $p_1=p_2=2$, or $p_2=p_3=2$.
In the first case, we dominate the integral by
\begin{equation}\label{eqn:1term}
\|v_k\|_{L^\infty}\cdots\|v_3\|_{L^\infty}\int K(y_2-x)\,v_2(y_2)\,v_1(x)\,dy_2\,dx
\end{equation}
where
$$
K(z)=\int G_{r,t}(z,y_3,\ldots,y_k)\,dy_3\cdots dy_k\,.
$$
Since $\int K=1$, by Young's inequality the integral in \eqref{eqn:1term} is bounded by $\|v_2\|_{L^2}\|v_1\|_{L^2}$. 

In case $p_2=p_3=2$, we bound the integral by
\begin{equation}\label{eqn:2term}
\|v_k\|_{L^\infty}\cdots\|v_4\|_{L^\infty}\|v_1\|_{L^\infty}\int K(y_2,y_3)\,v_3(y_3-x)\,v_2(y_2-x)\,dy_2\,dy_3\,dx\,,
\end{equation}
where now
$$
K(y_2,y_3)=\int G_{r,t}(y_2,y_3,y_4,\ldots,y_k)\,dy_4\cdots dy_k\,.
$$
Thus $\widehat K(\eta_2,\eta_3)=e^{-tQ_r(\eta_2,\eta_3,0,\ldots,0)}$.
Writing $v_2$ and $v_3$ in terms of their Fourier transforms, and integrating out $y_2$ and $y_3$, expresses the integral in \eqref{eqn:2term} as
\begin{multline*}
(2\pi)^{-2n}\int e^{-ix(\eta_2+\eta_3)}e^{-tQ_r(-\eta_2,-\eta_3,0,\ldots,0)}\,\widehat{v_3}(\eta_3)\,\widehat{v_2}(\eta_2)\,d\eta_2\,d\eta_3\,dx\\
=
(2\pi)^{-n}\int e^{-tQ_r(-\eta_2,\eta_2,0,\ldots,0)}\widehat{v_3}(-\eta_2)\widehat{v_2}(\eta_2)\,d\eta_2\,,
\end{multline*}
which is bounded by $\|v_3\|_{L^2}\|v_2\|_{L^2}$ by the Schwarz inequality, as $Q_r\ge 0$.

It remains to show the expansion holds for general $V\in L^\infty(\RR^n,\RR)\cap H^m(\RR^n)$. We set $\phi_\epsilon*V=V_\epsilon\in C_{\comp}^\infty(\RR^n)$, where $\phi_\epsilon=\epsilon^{-n}\phi(\epsilon^{-1}\cdot)$ is a family of smooth compactly supported mollifiers. 

Recall that $\tr\bigl(W_k(t)\bigr)=(4\pi t)^{-n/2}t^k B_t(V)$. Since for each $t$, $B_t(V)$ is continuous in $V$ in the $L^k(\RR^n)$ topology, then 
$B_t(V)=\lim_{\epsilon\rightarrow 0}B_t(V_\epsilon)$. Furthermore, since
$\|V_\epsilon\|_{L^\infty}\le \|V\|_{L^\infty}\,,$ $\|V_\epsilon\|_{H^m}\le \|V\|_{H^m}\,,$ we have the following bounds, uniform for $t>0$ and $\epsilon>0$,
$$
\|r_{k,k+m}(t,V_\epsilon)\|\le C\,\|V\|_{L^\infty}^{k-2}\|V\|_{H^m}^2\,.
$$
It thus remains to show that $\lim_{\epsilon\rightarrow 0}c_{k,k+j}(V_\epsilon)=c_{k,k+j}(V)$ if $j\le m-1$, for appropriately defined $c_{k,k+j}(V)$ satisfying the bounds of Proposition \ref{thm:traceWk}.

Recall that $c_{k,k+j}(V_\epsilon)$ can be written as a linear combination of terms of the form
\begin{equation}\label{eqn:ckjterm}
\int (\partial^{\alpha_k}V_\epsilon)(y)\cdots (\partial^{\alpha_1}V_\epsilon)(y)\,dy\,,
\end{equation}
where $|\alpha_i|\le j$ for all $i$, and $\sum_{i=1}^k|\alpha_i|=2j\,.$
We define $c_{k,k+j}(V)$ by the same formula, which by Lemma \ref{lem:GNM} is well defined, and absolutely dominated by $\|V\|_{L^\infty}^{k-2}\|D^jV\|_{L^2}^2\,.$ To see that \eqref{eqn:ckjterm} converges, as $\epsilon\rightarrow 0$, to the same expression with $V_\epsilon$ replaced by $V$, we note that, by the proof of Lemma \ref{lem:GNM}, 
$\partial^{\alpha_i} V\in L^{\frac{2m}{|\alpha_i|}}$, so for $|\alpha_i|>0$,
$$
\lim_{\epsilon\rightarrow 0}\|\partial^{\alpha_i} V_\epsilon-\partial^{\alpha_i} V\|_{L^{\frac{2m}{|\alpha_i|}}}=0\,.
$$
Thus, the product of the $\partial^{\alpha_i}V_\epsilon$ in \eqref{eqn:ckjterm} with $|\alpha_i|\ne 0$ converges in $L^{\frac mj}$ to the same product with $V_\epsilon$ replace by $V$.
Since $\frac mj>1$, the integral in \eqref{eqn:ckjterm} converges as $\epsilon\rightarrow 0$ by the fact that $V_\epsilon\rightarrow V$ in $L^p$ for all $p<\infty$.
    
\def\arXiv#1{\href{http://arxiv.org/abs/#1}{arXiv:#1}}


\begin{thebibliography}{0}

\bibitem[BaSa95]{bs} R. Ba\~nuelos and A. S\'a Barreto, 
{\em On the heat trace of Schr\"odinger operators}, Communications
in Partial Differential Equations, {\bf 20}(1995), 2153--2164.

\bibitem[vdB91]{vdb} M. van den Berg,
{\em On the trace of the difference of Schr\"odinger heat semigroups,}
Proc. Roy. Soc. Edinburgh Sect. A 119 (1991), no. 1-2, 161--175.

\bibitem[Br84]{Br} J. Br\"uning, {\em 
On the compactness of isospectral potentials}, Communications in Partial
Differential Equations, {\bf 9}(1984), 687--698.

\bibitem[Ch99]{Ch}
T. Christiansen, {\em Some lower bounds on the number of resonances in Euclidean scattering}, Math. Res. Letters {\bf 6}(1999), 203--211.

\bibitem[Ch05]{Ch0} T. Christiansen, {\em Several complex variables and the distribution of resonances in potential scattering.} Comm. Math. Phys. {\bf 259}(2005),  711--728.

\bibitem[Ch06]{Ch1}
T. Christiansen,
{\em Schr\"odinger operators with complex-valued potentials and no resonances.} 
Duke Math. J. {\bf 133}(2006), 313--323. 


\bibitem[ChHi05]{CH} T. Christiansen and P. Hislop, {\em The resonance counting function for Schr\"odinger operators with generic potentials,} Math. Res. Lett. 
{\bf 12}(2005), 821--826.

\bibitem[ChHi10]{CH1} T. Christiansen and P. Hislop, 
{\em Maximal order of growth for the resonance counting functions for generic potentials in even dimensions.} Indiana Univ. Math. J. {\bf 59}(2010),  621--660. 

\bibitem[CdV81]{cdv1} Y. Colin de Verdi\`ere, {\em Une formule de traces pour l'op\'erateurs de Schr\"odinger dans $R^3$,}
Ann.~Sci.~\'Ec.~Norm.~Sup. 4\'eme s\'erie, {\bf 14}(1981), 27--39.

\bibitem[CdV12]{cdv} Y. Colin de Verdi\`ere, {\em Semiclassical trace formulas 
and heat expansions,} Anal. PDE {\bf 5}(2012), 693--703.

\bibitem[DiVu13]{dv13} T.-C. Dinh and D.-V. Vu,
{\em Asymptotic number of scattering resonances for generic Schr\"dingier operators,}
Comm.~Math.~Phys.
326(2014), 185--208.


\bibitem[Do04]{Don} H. Donnelly, Compactness of isospectral potentials,
Trans.~A.M.S. {\bf 357}(2005), 1717-1730.

\bibitem[DyZw]{res} S. Dyatlov and M. Zworski,
		\emph{Mathematical theory of scattering resonances,\/}
		book in preparation; \url{http://math.berkeley.edu/\~zworski/res.pdf}
 
\bibitem[Gil04]{Gil} P. B. Gilkey, {\em Asymptotic formulae in spectral geometry}, CRC, Boca Raton, FL, 2004.
 
\bibitem[Gu81]{Gui} L. Guillop\'e, {\em Asymptotique de la phase de diffusion pour 
l'op\'erateur de Schr\"odinger avec potentiel,} C. R. Acad. Sci. Paris S\'er. I 
Math. {\bf 293}(1981), 601--603. 
 
 \bibitem[HiPo03]{PoHi} M. Hitrik and I. Polterovich, {\em 
Regularized traces and Taylor expansions for the heat semigroup,}
J. London Math. Soc. {\bf 68}(2003), 402--418.

\bibitem[Je90]{J}
A. Jensen, 
{\em High energy asymptotics for the total scattering phase in potential scattering theory.} Functional-analytic methods for partial differential equations (Tokyo, 1989), 187--195, Lecture Notes in Math., {\bf 1450}, Springer, Berlin, 1990. 

\bibitem[JeKa79]{JK} A. Jensen and T. Kato, {\em Spectral properties of          
Schr\"odinger operators and time-decay of the wave functions,} Duke Math. J.
{\bf 46}(1979), 583-611.
 
\bibitem[McMo75]{MM} H. McKean and P. van Moerbeke,
{The Spectrum of Hill's Equation}, Invent.~Math. {\bf 30}(1975), 217--274.

\bibitem[Me95]{Mel} R.B. Melrose, {\em Geometric Scattering Theory,}
Cambridge University Press 1995.


\bibitem[Sa01]{Sa} A. S\'a Barreto, 
{\em Remarks on the distribution of resonances in odd dimensional Euclidean
scattering,} Asymptot. Anal. {\bf 27}(2001), 161--170. 

\bibitem[SaZw96]{SZ} A. S\'a Barreto and M. Zworski,
{\em Existence of resonances in potential scattering.} 
Comm. Pure Appl. Math. {\bf 49}(1996), 1271--1280.

\bibitem[Tay11]{pde} M.E. Taylor, 
{\em Partial Differential Equations III: Nonlinear Equations,}
Applied Mathematical Sciences, {\bf 118}, 2nd edition, Springer 2011.

\bibitem[Ti64]{Tit} E.C. Titchmarsh, {\em The theory of functions}, 
2nd edition, Oxford University Press, 1964.


\bibitem[Zw89]{z} M. Zworski, 
{\em Sharp polynomial bounds on the number of scattering poles,}
Duke Math. J. {\bf 59}(1989), 311--323.

\bibitem[Zw97]{Zw} 
 M. Zworski, {\em Poisson formula for resonances,}
S\'eminaire E.D.P. (1996--1997), Expos\'e no XIII, 12 p.

\bibitem[Zw12]{e-z} M. Zworski,
        \emph{Semiclassical analysis,\/}
        Graduate Studies in Mathematics \textbf{138}, AMS, 2012.

\end{thebibliography}
\end{document}